\documentclass[12pt,a4paper,reqno]{amsart}

\usepackage{latexsym}
\usepackage{amsmath}
\usepackage{amsfonts}
\usepackage{a4wide}
\usepackage{amssymb}
\usepackage{amsthm}
\usepackage[colorlinks, citecolor=blue, linkcolor=red]{hyperref}

\theoremstyle{plain}

\newtheorem{teo}{Theorem}[section]
\newtheorem{lemma}[teo]{Lemma}
\newtheorem{prop}[teo]{Proposition}
\newtheorem{cor}[teo]{Corollary}
\newtheorem{ackn}{Acknowledgments\!}

\theoremstyle{definition}

\theoremstyle{remark}

\newtheorem{rem}[teo]{Remark}

\numberwithin{equation}{section}

\def\rd{\overset{\circ}{Ric}}
\def\rdc{\overset{\circ}{R}}
\def\kn{\mathbin{\bigcirc\mkern-15mu\wedge}}

\def\SS{{{\mathbb S}}}

\def\RR{{\mathbb R}}

\def\e{{\varepsilon}}

\def\eps{\varepsilon}

\title[Rigidity of positively curved shrinking Ricci solitons in dimension four]{Rigidity of positively curved\\shrinking Ricci solitons in dimension four}

\author[Giovanni Catino]{Giovanni Catino}
\address[Giovanni Catino]{Dipartimento di Matematica, Politecnico di Milano, Piazza Leonardo da Vinci 32, 20133 Milano, Italy}
\email[]{giovanni.catino@polimi.it}

\date{\today}

\begin{document}

\begin{abstract} 
We classify four-dimensional shrinking Ricci solitons satisfying $Sec \geq \frac{1}{24} R$, where $Sec$ and $R$ denote the sectional and the scalar curvature, respectively. They are isometric to either $\mathbb{R}^{4}$ (and quotients), $\mathbb{S}^{4}$, $\mathbb{RP}^{4}$ or $\mathbb{CP}^{2}$ with their standard metrics.
\end{abstract}

\maketitle

\begin{center}


\medskip

\centerline{\bf AMS subject classification:  53C24, 53C25}

\end{center}

\

\

\section{Introduction}

In this paper we investigate gradient shrinking Ricci solitons with positive sectional curvature. We recall that a Riemannian manifold $(M^n ,g)$ of dimension $n\geq 3$ is a gradient Ricci soliton if there exists a smooth function $f$ on $M^n$ such that 
$$
Ric + \nabla^2 f  =  \lambda\, g
$$ 
for some constant $\lambda$. If $\nabla f$ is parallel, then $(M^n,g)$ is Einstein. The Ricci soliton is called  shrinking if $\lambda>0$, steady if $\lambda=0$ and  expanding if $\lambda<0$. Ricci solitons generate self-similar solutions of the Ricci flow, play a fundamental role in the formation of singularities and have been studied by many authors (see H.-D. Cao~\cite{cao2} for an overview).

It is well known that (compact) Einstein manifolds can be classified, if they are enough positively curved. Sufficient conditions are non-negative curvature operator (S. Tachibana \cite{tachib1}), non-negative isotropic curvature (M. J. Micallef and Y. Wang \cite{micwan} in dimension four and S. Brendle \cite{bre} in every dimension) and weakly $\frac14$-pinched sectional curvature \cite{berger} (if $Sec$ and $R$ denote the sectional and the scalar curvature, respectively, this condition in dimension four is implied by $Sec \geq \frac{1}{24}R$). Moreover, in dimension four, it is proved by D. Yang \cite{yang}) that  four-dimensional Einstein manifolds satisfying $Sec \geq  \eps R$ are isometric to either $\mathbb{S}^{4}$, $\mathbb{RP}^{4}$ or $\mathbb{CP}^{2}$ with their standard metrics, if $\eps=\frac{\sqrt{1249}-23}{480}$. The lower bound has been improved to $\eps=\frac{2-\sqrt{2}}{24}$ by E. Costa \cite{costa} and, more recently, to $\eps=\frac{1}{48}$ by  E. Ribeiro \cite{rib} (see also X. Cao and P. Wu \cite{xcaowu}). It is conjectured in \cite{yang} that the result should be true assuming positive sectional curvature.

In dimension $n\leq 3$, complete shrinking Ricci solitons are classified. In the last years there have been a lot of interesting results concerning the classification of shrinking Ricci solitons which are positively curved. For instance, it follows by the work of C. B\"ohm and B. Wilking~\cite{bohmwilk} that the only compact shrinking Ricci solitons with positive (two-positive) curvature operator are quotients of $\SS^{n}$. In dimension four, A. Naber \cite{naber} classified complete shrinkers with non-negative curvature operator. Four dimensional shrinkers with non-negative isotropic curvature were classified by X. Li, L. Ni and K. Wang \cite{lnw}. 

Recently, O. Munteanu and J.P. Wang \cite{munwang} showed that every complete shrinking Ricci solitons with positive sectional curvature are compact. It is natural to ask the following question: {\em given $\eps>0$, are there four dimensional non-Einstein shrinking Ricci solitons satisfying $Sec \geq \eps R$?}.

\

In this paper we give an answer to this question proving the following

\begin{teo}\label{t-main} Let $(M^{4},g)$ be a four-dimensional complete gradient shrinking Ricci soliton with $Sec\geq\frac{1}{24} R$. Then $(M^{4},g)$ is necessarily Einstein, thus isometric to either $\RR^{4}$ (and quotients), $\SS^{4}$, $\mathbb{RP}^{4}$ or $\mathbb{CP}^2$ with their standard metrics.
\end{teo}

Note that, by the work of S. Brendle and R. Schoen \cite{bs}, using the Ricci flow, one can show that compact Ricci shrinkers with weakly $\frac14$-pinched sectional curvature are isometric to $\SS^{4}$, $\mathbb{RP}^{4}$ or $\mathbb{CP}^2$ with their standard metrics. The condition $Sec\geq\frac{1}{24} R$ is a little stronger, but the proof of Theorem \ref{t-main} that we present is completely ``elliptic''.

\

\section{Estimates on manifolds with positive sectional curvature} 

To fix the notation we recall that the Riemann curvature operator of a Riemannian manifold $(M^n,g)$ is defined as in~\cite{gahula} by
$$
Rm(X,Y)Z=\nabla_{Y}\nabla_{X}Z-\nabla_{X}\nabla_{Y}Z+\nabla_{[X,Y]}Z\,.
$$ 
In a local coordinate system the components of the $(3,1)$--Riemann curvature tensor are given by
$R^{l}_{ijk}\frac{\partial}{\partial
  x^{l}}=Rm\big(\frac{\partial}{\partial
  x^{i}},\frac{\partial}{\partial
  x^{j}}\big)\frac{\partial}{\partial x^{k}}$ and we denote by
$R_{ijkl}=g_{lp}R^{p}_{ijk}$ its $(4,0)$--version. Throughout the paper the Einstein convention of summing over the repeated indices will be adopted. The Ricci tensor $Ric$ is obtained by the contraction $(Ric)_{ik}=R_{ik}=g^{jl}R_{ijkl}$, $R=g^{ik}R_{ik}$ will denote the scalar curvature and $(\rd)_{ik}=R_{ik}-\frac{1}{n}R\, g_{ik}$ the traceless Ricci tensor. The Riemannian metric induces norms on all the tensor bundles, in coordinates this norm is given, for a tensor $T=T_{i_{1}\dots i_{k}}^{j_{1}\dots j_{l}}$, by
$$
|T|^{2}_{g}=g^{i_{1}m_{1}}\cdots g^{i_{k}m_{k}}
 g_{j_{1}n_{1}}\dots g_{j_{l}n_{l}} T_{i_{1}\dots i_{k}}^{j_{1}\dots j_{l}} T_{m_{1}\dots m_{k}}^{n_{1}\dots n_{l}}\,.
$$

The first key observation are the following pointwise estimates which are satisfied by every metric with $Sec \geq \eps R$ for some $\eps \in \RR$. 
\begin{prop}\label{p-est}
Let $(M^{n},g)$ be a Riemannian manifold of dimension $n\geq 3$. If the sectional curvature satisfies $Sec\geq \e R$ for some $\e\in\RR$, then the following two estimates hold
$$
R_{ijkl} \rdc_{ik} \rdc_{jl} \leq \frac{1-n^{2}\e}{n} R |\rd|^{2} +\rdc_{ij}\rdc_{ik}\rdc_{jk}
$$
and
$$
R_{ijkl} \rdc_{ik} \rdc_{jl} \leq \frac{n^{2}-4n+2-n^{2}(n-2)(n-3)\e}{2n} R |\rd|^{2} -(n-1)\rdc_{ij}\rdc_{ik}\rdc_{jk}
 \,.
$$
In particular, in dimension four
$$
R_{ijkl} \rdc_{ik} \rdc_{jl} \leq \frac{1-16\e}{4} R |\rd|^{2} +\rdc_{ij}\rdc_{ik}\rdc_{jk}
 \,,
$$
$$
R_{ijkl} \rdc_{ik} \rdc_{jl} \leq \frac{1-16\e}{4} R |\rd|^{2} -3\rdc_{ij}\rdc_{ik}\rdc_{jk}
 \,.
$$

\end{prop}
\begin{proof} Let $\{e_{i}\}$, $i=1,\ldots, n$, be the eigenvectors of $\rd$ and let $\lambda_{i}$ be the corresponding eigenvalues. Moreover, let $\sigma_{ij}$ be the sectional curvature defined by the two-plane spanned by $e_{i}$ and $e_{j}$. Since the sectional curvature satisfy $Sec \geq \e R$, it is natural to define the tensor
$$
\overline{Rm} = Rm - \frac{\e}{2}R\, g\kn g \,.
$$
In particular
$$
\overline{Ric}=Ric-(n-1)\e R\,g\,,\quad\quad \overline{R}=\big(1-n(n-1)\e\big) R \quad\quad\hbox{and}\quad \quad \overline{\sigma}_{ij}=\sigma_{ij}-\e R \geq 0\,.
$$
Moreover, if  $\mu_{k}$ and $\overline{\mu}_{k}$ are the eigenvalues with eigenvector $e_{k}$ of $Ric$ and $\overline{Ric}$, respectively, one has
$$
\mu_{k} = \sum_{i\neq k} \sigma_{ik} \quad\quad\hbox{and}\quad\quad \overline{\mu}_{k} = \sum_{i\neq k} \overline{\sigma}_{ik} \,.
$$
Denoting by $\overline{R}_{ijkl}$ the components of $\overline{Rm}$, we get
\begin{align}\label{equno}
\overline{R}_{ijkl} \rdc_{ik} \rdc_{jl}  - \overline{R}_{ij}\rdc_{ik}\rdc_{jk} &= \sum_{i,j=1}^{n} \lambda_{i}\lambda_{j} \overline{\sigma}_{ij} - \sum_{k=1}^{n}\mu_{k}\lambda_{k}^{2} \nonumber\\
&= 2\sum_{i<j} \lambda_{i}\lambda_{j} \overline{\sigma}_{ij} - \sum_{i<j}\big(\lambda_{i}^{2}+\lambda_{j}^{2}\big)\overline{\sigma}_{ij} \nonumber\\
&= - \sum_{i<j}\big(\lambda_{i}-\lambda_{j}\big)^{2}\overline{\sigma}_{ij} \leq 0 \,.
\end{align}
Using the definition of $\overline{Rm}$ and $\overline{Ric}$, we obtain
$$
R_{ijkl} \rdc_{ik} \rdc_{jl} + \e R |\rd|^{2} \leq R_{ij}\rdc_{ik}\rdc_{jk} -(n-1)\e R |\rd|^{2} = \rdc_{ij}\rdc_{ik}\rdc_{jk}+\frac{1-n(n-1)\e}{n} R |\rd|^{2}
$$
and this proves the first inequality of this proposition. 

\

\noindent In order to show the second one, we will follow the proof of \cite[Proposition 3.1]{cat}. We observe that
$$
\overline{R}_{ikjl} \rdc_{ij} \rdc_{kl}  - \frac{n-2}{2n} \overline{R}|\rd|^{2} \,=\, \sum_{i,j=1}^{n} \lambda_{i}\lambda_{j} \overline{\sigma}_{ij} - \frac{n-2}{2n} \overline{R} \sum_{k=1}^{n} \lambda_{k}^{2} \,.
$$
Since the modified scalar curvature $\overline{R}$ can be written as
$$
\overline{R} \,=\, g^{ij}g^{kl}\overline{R}_{ikjl} \,=\, \sum_{i,j=1}^{n} \overline{\sigma}_{ij} \,=\, 2 \sum_{i<j}\overline{\sigma}_{ij}\,,
$$
one has the following
\begin{align*} 
\sum_{i,j=1}^{n} \lambda_{i}\lambda_{j} \overline{\sigma}_{ij} - \frac{n-2}{2n}  \overline{R} \sum_{k=1}^{n} \lambda_{k}^{2} &= 2 \sum_{i<j}\lambda_{i}\lambda_{j} \overline{\sigma}_{ij} - \frac{n-2}{n}  \sum_{i<j} \overline{\sigma}_{ij} \sum_{k=1}^{n} \lambda_{k}^{2} \\
&= \sum_{i<j} \Big(2\lambda_{i}\lambda_{j} -\frac{n-2}{n}\sum_{k=1}^{n} \lambda_{k}^{2} \Big) \overline{\sigma}_{ij}\,. 
\end{align*}
On the other hand, one has 
$$
\sum_{k=1}^{n} \lambda_{k}^{2} \,=\, \lambda_{i}^{2} + \lambda_{j}^{2} + \sum_{k\neq i,j}\lambda_{k}^{2}\,.
$$
Moreover, using the Cauchy-Schwarz inequality and the fact that $\sum_{k=1}^{n} \lambda_{k} =0$, we obtain
$$
\sum_{k\neq i,j}\lambda_{k}^{2} \,\geq \, \frac{1}{n-2} \Big( \sum_{k\neq i,j}\lambda_{k}\Big)^{2} \,=\, \frac{1}{n-2} \big( \lambda_{i} + \lambda_{j} \big)^{2} 
$$
with equality if and only if $\lambda_{k}=\lambda_{k'}$ for every $k,k'\neq i,j$. Hence, the following estimate holds 
$$
\sum_{k=1}^{n} \lambda_{k}^{2} \,\geq \, \frac{n-1}{n-2} \big(\lambda_{i}^{2} + \lambda_{j}^{2} \big) + \frac{2}{n-2} \lambda_{i} \lambda_{j}\,.
$$
Using this, since $\overline{\sigma}_{ij} \geq 0$, it follows that
\begin{align*} 
\sum_{i,j=1}^{n} \lambda_{i}\lambda_{j} \overline{\sigma}_{ij} - \frac{n-2}{2n} \overline{R} \sum_{k=1}^{n} \lambda_{k}^{2} &\leq \frac{n-1}{n}\sum_{i<j} \big(2\lambda_{i}\lambda_{j} -\big(\lambda_{i}^{2} + \lambda_{j}^{2} \big) \Big) \overline{\sigma}_{ij} \\
&=- \frac{n-1}{n}\sum_{i<j} (\lambda_{i}-\lambda_{j})^{2} \overline{\sigma}_{ij} \\
&= \frac{n-1}{n}\Big( \overline{R}_{ijkl} \rdc_{ik} \rdc_{jl}  - \overline{R}_{ij}\rdc_{ik}\rdc_{jk} \Big)\,,
\end{align*}
where in the last equality we have used equation \eqref{equno}. Hence, we proved
$$
\overline{R}_{ikjl} \rdc_{ij} \rdc_{kl}  - \frac{n-2}{2n} \overline{R}|\rd|^{2} \leq \frac{n-1}{n}\Big( \overline{R}_{ijkl} \rdc_{ik} \rdc_{jl}  - \overline{R}_{ij}\rdc_{ik}\rdc_{jk} \Big) \,,
$$
i.e.
$$
\overline{R}_{ikjl} \rdc_{ij} \rdc_{kl} \leq \frac{n-2}{2} \overline{R}|\rd|^{2} - (n-1) \overline{R}_{ij}\rdc_{ik}\rdc_{jk} \,.
$$
Finally, substituting  $\overline{Rm}$, $\overline{Ric}$ and $\overline{R}$ we obtain the the second inequality of this proposition.

\end{proof}

Taking the convex combination of the two previous estimates we obtain the following.

\begin{cor}\label{c-est} Let $(M^{n},g)$ be a Riemannian manifold of dimension $n\geq 3$. If the sectional curvature satisfies $Sec\geq \e R$ for some $\e\in\RR$, then, for every $s\in[0,1]$, one has
\begin{align*}
R_{ijkl} \rdc_{ik} \rdc_{jl} \leq& \left(\frac{n^{2}-4n+2-n^{2}(n-2)(n-3)\e}{2n}-\frac{n-4}{2}\big(1-n(n-1)\e\big)s\right) R |\rd|^{2}
\\&\, -(n-1-ns)\rdc_{ij}\rdc_{ik}\rdc_{jk}
 \,.
\end{align*}
In particular, in dimension four, for every $s\in[0,1]$, one has
$$
R_{ijkl} \rdc_{ik} \rdc_{jl} \leq \frac{1-16\e}{4} R |\rd|^{2} -(3-4s)\rdc_{ij}\rdc_{ik}\rdc_{jk}
 \,,
$$
\end{cor}

\begin{rem} Taking $\eps=0$ and $s=\frac{n-1}{n}$, we recover the estimate on manifolds with non-negative sectional curvature which was proved in \cite{cat}.
\end{rem}

\

\section{Some formulas for Ricci solitons}

Let $(M^{n},g)$ be a $n$-dimensional complete gradient shrinking Ricci solitons 
$$
Ric + \nabla^{2} f \,=\, \lambda g 
$$
for some smooth function $f$ and some positive constant $\lambda>0$. First of all we recall the following well known formulas (for the proof see~\cite{mantemin2}) 

\begin{lemma}\label{formulas} Let $(M^{n},g)$ be a gradient Ricci soliton. Then the following formulas hold
\begin{equation*}\label{eq0}
\Delta f \,=\, n \lambda - R 
\end{equation*}
\begin{equation*}\label{eq1}
\Delta_{f} R \,= \, 2\lambda R-2|Ric|^2
\end{equation*}
\begin{align*}
\Delta_{f} R_{ik}
\,=\,&\,2\lambda R_{ik}-2R_{ijkl} R_{jl}\label{eq2}
\end{align*}
where the $\Delta_{f}$ denotes the $f$-Laplacian, $\Delta_{f}=\Delta-\nabla_{\nabla f}$.
\end{lemma}

In particular, defining $\rdc_{ij}=R_{ij}-\frac{1}{n}R\, g_{ij}$, a simple computation shows the following equation for the $f$-Laplacian of the squared norm of the trece-less Ricci tensor $\rd$
\begin{lemma}\label{l-eqric} Let $(M^{n},g)$ be a gradient Ricci soliton. Then the following formula holds
\begin{equation*}
\frac{1}{2}\Delta_{f} |\rd|^{2} \,=\, |\nabla \rd|^{2} + 2\lambda |\rd|^{2} -2R_{ijkl}\rdc_{ik}\rdc_{jl}-\frac{2}{n}R|\rd|^{2}\,.
\end{equation*}
\end{lemma}

Moreover we have the following scalar curvature estimate \cite{pirise}.

\begin{lemma}\label{l-scal} Let $(M^{n},g)$ be a complete gradient shrinking Ricci soliton. Then either $g$ is flat or its scalar curvature is positive $R>0$.
\end{lemma}

Finally, we show this simple identity.

\begin{lemma}\label{l-id} Let $(M^{n},g)$ be a compact gradient Ricci soliton. Then the following formula holds
$$
\int |\nabla R|^2\,dV = \frac{n-4}{2}\,\lambda\int R^{2}\,dV -\frac{n-4}{2n}\int R^{3}\,dV + 2 \int R |\rd|^{2}\,dV \,. 
$$
In particular, in dimension four
$$
\int |\nabla R|^2\,dV = 2 \int R |\rd|^{2}\,dV \,. 
$$
\end{lemma}
\begin{proof} Integrating by parts and using Lemma \ref{formulas} we obtain
\begin{align*}
\int |\nabla R|^{2}\,dV &= -\int R\Delta R\,dV \\
&= -\frac{1}{2}\int \langle \nabla R^{2}, \nabla f\rangle\,dV - 2\lambda \int R^{2}\,dV + 2 \int R |\rd|^{2}\,dV + \frac{2}{n} \int R^{3}\,dV  \\
&= \frac{1}{2}\int R^{2} \Delta f\,dV - 2\lambda \int R^{2}\,dV + 2 \int R |\rd|^{2}\,dV + \frac{2}{n} \int R^{3}\,dV \\
&= \frac{n-4}{2}\,\lambda\int R^{2}\,dV -\frac{n-4}{2n}\int R^{3}\,dV + 2 \int R |\rd|^{2}\,dV \,.
\end{align*}
\end{proof}

\

\section{Proof of Theorem \ref{t-main}}

Let $(M^4,g)$ be a complete gradient shrinking Ricci soliton of dimension four and assume that $Sec\geq \eps R$ on $M^4$ for some $\eps>0$. By Lemma \ref{l-scal} either $g$ is flat or $R>0$. In this second case, by the result in \cite{munwang} we know that $M^4$ must be compact. From now on we can assume that $(M^{4},g)$ is compact with $Sec\geq \eps R>0$. Lemma \ref{l-eqric} gives
\begin{align*}
\frac{1}{2}\Delta_{f} |\rd|^{2} &= |\nabla \rd|^{2} + 2\lambda |\rd|^{2} -2R_{ijkl}\rdc_{ik}\rdc_{jl}-\frac{1}{2}R|\rd|^{2}\,.
\end{align*}
Integrate over $M^{4}$ and using equation \eqref{eq0} we obtain
\begin{align}\label{est0} \nonumber
0 &= \frac{1}{2}\int\langle \nabla |\rd|^{2},\nabla f\rangle\,dV +\int|\nabla \rd|^{2}\,dV+ 2\int\lambda |\rd|^{2}\,dV\\\nonumber
&\quad-2\int R_{ijkl}\rdc_{ik}\rdc_{jl}\,dV-\frac{1}{2}\int R|\rd|^{2}\,dV\\
&= \int|\nabla \rd|^{2}\,dV-2\int R_{ijkl}\rdc_{ik}\rdc_{jl}\,dV \,. 
\end{align}
On the other hand, given $a_1,a_2,b_1,b_2,b_3\in\RR$ we define the three tensor
$$
F_{ijk}:=\nabla_k\rdc_{ij}+a_1 \nabla_j\rdc_{ik}+a_2 \nabla_i\rdc_{jk} + b_1 \nabla_k R g_{ij} + b_2 \nabla_j R g_{ik} + b_3 \nabla_i R g_{jk} \,.
$$
Using the Bianchi identity $\nabla_i \rdc_{ij} = \frac{1}{4} \nabla_j R$, a computation gives
\begin{align*}
&|F|^2 = (1+a_1^2+a_2^2)|\nabla \rd|^2 + 2(a_1+a_2+a_1a_2)\nabla_k \rdc_{ij} \nabla_j \rdc_{ik} \\
&+\frac{1}{2}\left( a_1 (b_1 +  b_3) + a_2 (b_1 + b_2) + b_2+b_3+8 ( b_1^2+b_2^2+b_3^2)+4 (b_1b_2+b_1b_3+b_2b_3)\right)|\nabla R|^2 \,.
\end{align*}
In particular, 
\begin{align}\label{eq-est1} \nonumber
&\int |\nabla \rd|^2\,dV = \frac{1}{1+a_1^2+a_2^2}\int|F|^{2}\,dV-\frac{2(a_1+a_2+a_1a_2)}{1+a_1^2+a_2^2}\int \nabla_k \rdc_{ij} \nabla_j \rdc_{ik}\,dV \\\nonumber
&-\frac{a_1 (b_1 +  b_3) + a_2 (b_1 + b_2) + b_2+b_3+8 ( b_1^2+b_2^2+b_3^2)+4 (b_1b_2+b_1b_3+b_2b_3)}{2(1+a_1^2+a_2^2)}\int|\nabla R|^2\,dV  \\
&\hspace{2.6cm}=\frac{1}{1+a_1^2+a_2^2}\int|F|^{2}\,dV-\frac{2(a_1+a_2+a_1a_2)}{1+a_1^2+a_2^2}\int \nabla_k \rdc_{ij} \nabla_j \rdc_{ik}\,dV \\\nonumber
&-\frac{a_1 (b_1 +  b_3) + a_2 (b_1 + b_2) + b_2+b_3+8 ( b_1^2+b_2^2+b_3^2)+4 (b_1b_2+b_1b_3+b_2b_3)}{1+a_1^2+a_2^2}\int R|\rd|^2\,dV \,,
\end{align}
where, in the last equality we have used Lemma \ref{l-id}. On the other hand, integrating by parts and commuting the covariant derivatives,  one has
\begin{align}\label{eq-est8}\nonumber
\int \nabla_{k} \rdc_{ij} \nabla_{j}\rdc_{ik}\,dV &= -\int \rdc_{ij} \nabla_{k} \nabla_{j} \rdc_{ik} \, dV\\\nonumber
&= - \int \Big( \rdc_{ij} \nabla_{j} \nabla_{k} \rdc_{ik} + R_{kjil}\rdc_{ij}\rdc_{kl} + R_{ij}\rdc_{ik}\rdc_{jl}\Big) dV \\\nonumber
&= - \int \left( \frac{1}{4}\rdc_{ij} \nabla_i \nabla_j R- R_{ijkl}\rdc_{ik}\rdc_{jl} + \rdc_{ij}\rdc_{ik}\rdc_{jl} + \frac{1}{4} R|\rd|^{2}\right) dV \\\nonumber
&=  \int \left( \frac{1}{16} |\nabla R|^{2} + R_{ijkl}\rdc_{ik}\rdc_{jl} - \rdc_{ij}\rdc_{ik}\rdc_{jl} - \frac{1}{4} R|\rd|^{2}\right) dV \\
&= \int \left( R_{ijkl}\rdc_{ik}\rdc_{jl} - \rdc_{ij}\rdc_{ik}\rdc_{jl} - \frac{1}{8} R|\rd|^{2}\right) dV \,.
\end{align}
From equation \eqref{eq-est1}, we obtain
\begin{align*}
\int |\nabla \rd|^2\,dV &=\frac{1}{1+a_1^2+a_2^2}\int|F|^{2}dV -\frac{2(a_1+a_2+a_1a_2)}{1+a_1^2+a_2^2}\int \left( R_{ijkl}\rdc_{ik}\rdc_{jl} - \rdc_{ij}\rdc_{ik}\rdc_{jl}\right)dV\\
&\quad +Q_1 \int R|\rd|^2\,dV 
\end{align*}
with
\begin{align*}
 Q_1&:=\frac{a_1+a_2+a_1a_2}{4(1+a_1^2+a_2^2)}\\
&-\frac{a_1 (b_1 +  b_3) + a_2 (b_1 + b_2) + b_2+b_3+8 ( b_1^2+b_2^2+b_3^2)+4 (b_1b_2+b_1b_3+b_2b_3)}{1+a_1^2+a_2^2} \,.
\end{align*}
Using this inequality in \eqref{est0}, we obtain that
\begin{align}\label{est2}
0 &=  \frac{1}{1+a_1^2+a_2^2}\int|F|^{2}dV -\frac{2(1+a_1^2+a_2^2+a_1+a_2+a_1a_2)}{1+a_1^2+a_2^2}\int R_{ijkl}\rdc_{ik}\rdc_{jl}\,dV \\\nonumber
&\quad + \frac{2(a_1+a_2+a_1a_2)}{1+a_1^2+a_2^2} \int \rdc_{ij}\rdc_{ik}\rdc_{jl}\,dV+Q_1 \int R|\rd|^2\,dV \,.
\end{align}
From Corollary \ref{c-est} we have
\begin{equation}\label{eq-estr}
R_{ijkl} \rdc_{ik} \rdc_{jl} \leq \frac{1-16\e}{4} R |\rd|^{2} -(3-4s)\rdc_{ij}\rdc_{ik}\rdc_{jk}
\end{equation}
for every $s\in[0,1]$. Thus, if $a_1+a_2+a_1a_2\geq 0$, for every $s\in[0,1]$,  estimate \eqref{est2} gives
\begin{align}\label{est3}\nonumber
0 &\geq \frac{1}{1+a_1^2+a_2^2}\int|F|^{2}dV\\
&\quad+ \frac{2\Big((3-4s)(1+a_1^2+a_2^2)+4(1-s)(a_1+a_2+a_1a_2)\Big)}{1+a_1^2+a_2^2} \int \rdc_{ij}\rdc_{ik}\rdc_{jl}\,dV\\
&\quad+Q_2 \int R|\rd|^2\,dV
\end{align}
with
\begin{align*}
 Q_2 &:= Q_1-\frac{(1-16\eps)(1+a_1^2+a_2^2+a_1+a_2+a_1a_2)}{2(1+a_1^2+a_2^2)}\\
 &=\frac{a_1+a_2+a_1a_2}{4(1+a_1^2+a_2^2)}-\frac{(1-16\eps)(1+a_1^2+a_2^2+a_1+a_2+a_1a_2)}{2(1+a_1^2+a_2^2)}\\
&-\frac{a_1 (b_1 +  b_3) + a_2 (b_1 + b_2) + b_2+b_3+8 ( b_1^2+b_2^2+b_3^2)+4 (b_1b_2+b_1b_3+b_2b_3)}{1+a_1^2+a_2^2} \,.
\end{align*}
Now, choose $a_1=a_2=1$ and $b_1=b_2=b_3=:b$. Then
$$
Q_2 = -12 b^2 -2b+16\eps-\frac{3}{4} \,.
$$
In particular, the maximum is attained at $b=-1/12$ and is given by
\begin{equation}\label{max}
Q_2= \frac{48\eps-2}{3}\,.
\end{equation}
Actually a (long) computation gives that the maximum of the function $Q_2$ defined for general variables $(a_1,a_2,b_1,b_2,b_3)$ is attained at the point
\begin{equation}\label{choice}
(a_1,a_2,b_1,b_2,b_3)=\left(1,1,-\frac{1}{12},-\frac{1}{12},-\frac{1}{12}\right)
\end{equation}
and is given by the value \eqref{max}. Moreover, under the choice \eqref{choice}, one has
$$
\frac{2\Big((3-4s)(1+a_1^2+a_2^2)+4(1-s)(a_1+a_2+a_1a_2)\Big)}{1+a_1^2+a_2^2} = 2(7-8s)\,.
$$
In particular, choosing
$$
s=\frac{7}{8} \,,
$$
from \eqref{est3} we obtain
$$
0 \geq \frac{1}{3}\int|F|^{2}dV+\frac{48\eps-2}{3} \int R|\rd|^2\,dV \,.
$$
Thus, if $\eps>1/24$, then $\rd\equiv 0$, i.e. $(M^4,g)$ is Einstein. By Berger classification result \cite{berger} we conclude the proof of Theorem \ref{t-main} in this case. 

If $\eps=1/24$, then $Q_{1}=1/3$, $Q_{2}=0$ and all previous inequalities become equalities. In particular, $F\equiv 0$. Moreover, from \eqref{eq-estr}, we get 
\begin{equation}\label{eq127a}
R_{ijkl} \rdc_{ik} \rdc_{jl} \equiv \frac{1}{12} R |\rd|^{2} \quad\text{and}\quad \rdc_{ij}\rdc_{ik}\rdc_{jk}\equiv 0 \,.
\end{equation}
From equation \eqref{eq-est8} and Lemma \ref{l-id} we get
$$
\int \nabla_{k} \rdc_{ij} \nabla_{j}\rdc_{ik}\,dV = -\frac{1}{24}\int R|\rd|^{2}dV =  -\frac{1}{48}\int|\nabla R|^{2}dV.
$$
Thus, equation \eqref{eq-est1} gives
\begin{equation}\label{eq127b}
\int|\nabla \rd|^{2}dV = \frac{1}{12}\int|\nabla R|^{2}dV.
\end{equation}
Now, to conclude, we have to use the fact that $F\equiv 0$, i.e.
$$
0 = \nabla_k\rdc_{ij}+\nabla_j\rdc_{ik}+ \nabla_i\rdc_{jk} -\frac{1}{12}\left(\nabla_k R g_{ij} + \nabla_j R g_{ik} + \nabla_i R g_{jk}\right) \,.
$$
Taking the diverge in $k$ and contracting with $\rdc_{ij}$, we obtain
\begin{align*}
0 &= \rdc_{ij}\left[\Delta \rdc_{ij}+\nabla_{k}\nabla_j\rdc_{ik}+ \nabla_{k}\nabla_i\rdc_{jk} -\frac{1}{12}\left(\Delta R g_{ij} + 2\nabla_{i}\nabla_j R \right)\right]\\
&= \frac{1}{2}\Delta |\rd|^{2}-|\nabla \rd|^{2}+\rdc_{ij}\left[\nabla_{j}\nabla_k\rdc_{ik}+ \nabla_{i}\nabla_k\rdc_{jk} -\frac{1}{6}\nabla_{i}\nabla_j R \right]-2R_{ijkl}\rdc_{ik}\rdc_{jl}\\
&=\frac{1}{2}\Delta |\rd|^{2}-|\nabla \rd|^{2}+\frac13\rdc_{ij}\nabla_{i}\nabla_j R -2R_{ijkl}\rdc_{ik}\rdc_{jl}\\
&=\frac{1}{2}\Delta |\rd|^{2}-|\nabla \rd|^{2}+\frac13\rdc_{ij}\nabla_{i}\nabla_j R -\frac16R|\rd|^{2} \,,
\end{align*}
where we used \eqref{eq127a}. Integrating by parts over $M$, using \eqref{eq127b}, we obtain
$$
0 = -\frac16\int|\nabla R|^{2}dV-\frac16\int R|\rd|^{2}dV  
$$
which implies $\rd\equiv 0$, i.e. $(M^4,g)$ is Einstein and the thesis follows again by Berger result. This concludes the proof of Theorem \ref{t-main}.

\

\begin{ackn}
\noindent The author is member of the Gruppo Nazionale per
l'Analisi Matematica, la Probabilit\`{a} e le loro Applicazioni (GNAMPA) of the Istituto Nazionale di Alta Matematica (INdAM).
\end{ackn}

\

\

\bibliographystyle{amsplain}

\

\

\parindent=0pt

\end{document}